\documentclass[12pt]{amsart}
\usepackage{amscd, amssymb, amsmath}
\input{xy}
\xyoption{all}
\usepackage{hyperref}

\usepackage{geometry}
\usepackage{relsize}
\usepackage{type1cm}

\theoremstyle{definition}

\newtheorem{defi}{Definition}[section]

\newtheorem{rmk}[defi]{Remark}

\newtheorem{pt}{}[section]

\newtheorem{lem}[defi]{Lemma}
\newtheorem{thm}[defi]{Theorem}
\newtheorem{ptlem}{Lemma}[pt]

\newtheorem{thmm}{Theorem}
\newtheorem{coro}[thmm]{Corollary}

\newtheorem*{quest}{\underline{Question}}

\newtheorem*{set}{\underline{Setting}}

\theoremstyle{theorem}

\numberwithin{equation}{section}

\theoremstyle{rmk}

\newcommand{\xyz}{\xymatrix}

\newcommand{\ra}{\rangle}

\title{On the extension of holomorphic sections from reduced unions of strata of divisors}
\author{Chen-Yu Chi}
\thanks{
This work was partially supported by the Ministry of Science and Technology project (grant No.\ 107-2115-M-002-013 from August 2018) and by the European Research Council project “Algebraic and K\"ahler Geometry” (ERC-ALKAGE, grant No.\ 670846 from September 2015).  
}

\address{Chen-Yu Chi: Department of Mathematics,
National Taiwan University, Taipei 10617, Taiwan}
\email{chi@math.ntu.edu.tw}

\begin{document}
\fontsize{13pt}{20pt}\selectfont
\maketitle

\setlength{\baselineskip}{20pt}

\begin{abstract} We study the problem of extension of holomorphic sections of adjoint bundles from reduced unions of strata of divisors. The main technical result is a theorem of the Ohsawa--Takegoshi type, from which incorporated with the strong openness theorem we derive a qualitative extension theorem for adjoint bundles.
\end{abstract}
\section{{\bf Introduction}}
In the current work we study the possibility of extending holomorphic sections of an adjoint bundle on a complex manifold simultaneously from the reduced union of several strata of an effective divisor. Before describing the question to study and stating the main results we first fix terminology and notation which will be in force throughout the whole paper. In the following 
\begin{itemize}
\item[-] $X$ will be a connected complex manifold admitting a projective morphism to a Stein space, and
\item[-] $\mathcal S:(S_j,\sigma_j,h_j)\, (j=1,\dots,q)$ a family of data which consist of a holomorphic line bundle $S_j$ on $X$, a nonzero holomorphic section $\sigma_j$ of $S_j$, and a smooth hermitian metric $h_j$ on $S_j$ fro every $j$.
\end{itemize}
The Cartier divisor $\mathrm{div}(\sigma_j)$ will also be denoted by $S_j$. Note that we allow data with different indices to coincide. 
\begin{defi}
(1) For any $J\subseteq\{1,\dots,q\}$ we let $S_J=\bigcap_{j\in J}S_j$ (as an analytic subspace of $X$) and 
$
\big(S^J,\sigma^J,h^J\big)=\bigotimes\nolimits_{j\notin J}(S_j,\sigma_j,h_j)
$. We denote $(S^\emptyset,\sigma^\emptyset,h^\emptyset)$ by $(S,\sigma_S,h_S)$. For any $\alpha_\bullet=(\alpha_1,\dots,\alpha_q)\in\mathbf R^q$, we let
$$
|\sigma_S|_{h_S}^{2\alpha_\bullet}=\prod\limits_{j=1}^q|\sigma_j|_{h_j}^{2\alpha_j}\ \text{ and }\ |\sigma^J|_{h^J}^{2\alpha_\bullet}=\prod\limits_{j\notin J}|\sigma_j|_{h_j}^{2\alpha_j}\quad (J\subseteq \{1,\dots,q\}).
$$
(2) {\bf (Transversal strata)} A nonempty analytic set $W$ in $X$ is called an ($\mathcal S$){\it -transversal stratum} associated to some $J\subseteq\{1,\dots,q\}$ if $W$ is an irreducible component of $S_J$ such that there exists a nowhere dense analytic subset $W_1$ of $W$ satisfying the following conditions:  along $W\setminus W_1$ the divisors $S_j\, (j\in J)$ are smooth and intersect transversally, and $(W\setminus W_1)\cap \bigcup_{j\notin J}S_j=\emptyset$. If this is the case, we denote $J$ by $J_W$.\footnote{It is clear that if $j\in J_W$ for some admissible $W$ then $S_j\neq S_{j'}$ (as divisors) for every $j'\neq j$.}  
\\
(3) {\bf (Adjunction maps)} For an transversal stratum $W$ associated to $J$, we let $S^W$ denote the line bundle $S^J|_{W_{\mathrm{reg}}}$. For any holomorphic vector bundle $F$, we have the adjunction map
$$
\xyz{
\Gamma(X,K_X\otimes S\otimes F)
\ar[r]^-{(\cdot)_W}&
\Gamma(W_{\mathrm{reg}},K_{W_{\mathrm{reg}}}\otimes S^W\otimes F|_{W_{\mathrm{reg}}}),}
$$   
which depends at most on the choice of the defining sections $\sigma_j$ and on the order in which $S_j$ are listed.\\
(4) {\bf (Admissible families)} A family $\mathcal W$ of transversal strata is called ($\mathcal S$){\it-admissible} if there is no inclusion relation between the members of $\mathcal W$. For such a family we let
$$ 
J(\mathcal W):=\bigcup_{W\in\mathcal W}J_W
\ \text{ and }\
\mathrm{J}_0(\mathcal W):=\{J_W\,|\, W\in\mathcal W\}.
$$
(5) {\bf (Underlying spaces of admissible families)} For a family $\mathcal W$ we let 
$$
\underline{\mathcal W}:=\big(\bigcup\nolimits_{W\in\mathcal W} W\big)_{\mathrm{red}}.
$$
(6) {\bf (SNC families)} We say that an admissible family $\mathcal W$ is ($\mathcal S$-)snc if $S_1+\cdots+S_q$ has simple normal crossing in a neighborhood of $W$ for every $W\in\mathcal W$. (Note that all admissible families are snc when the divisor $S$ is reduced and of simple normal crossing.) 
\\
(7) {\bf (Derived families)} For an $\mathcal S$-snc family $\mathcal W$ we let $\mathcal W'$ be the family consisting of maximal members (with respect to inclusion) of 
$$
\left\{
W'\,
\left|
\begin{array}{c}
W'\text{ is a stratum of the snc divisor}\\
W\cap S^{J_W}\text{ on }W\text{ for some }W\in\mathcal W
\end{array}
\right.\right\}.
$$
$\mathcal W'$ is also an $\mathcal S$-snc family. More generally, we let $\mathcal W^{(k)}=(\mathcal W^{(k-1)})'$ for all $k\in\mathbf N$ ($\mathcal W^{(0)}:=\mathcal W$). 
\end{defi}

\begin{defi}\label{dominate a small multiple}
Given two real $(1,1)$-currents $T_1$ and $T_2$ on $X$, we say that $T_1$ locally dominates sufficiently small multiples of $T_2$ (Notation: $T_1\succcurlyeq \pm T_2$) if for every point $p\in X$ there exist a neighborhood $V$ of $p$ and a number $c_0>0$ such that $(T_1-c T_2)|_V$ is a positive $(1,1)$-current for every $c\in (-c_0,c_0)$. There is obviously a natural extension of this definition to endomorphism-valued $(1,1)$-currents in the sense of Nakano.
\end{defi} 
In this paper we prove the following extension theorem.
\begin{thmm}\label{from snc} 
Suppose that $\mathcal W$ is an $\mathcal S$-snc family and $(F,h_F)$ is a holomorphic vector bundle with a hermitian metric such that
$
\sqrt{-1}\, \Theta_{h_F}\succcurlyeq\pm\sqrt{-1}\, \Theta_{h_j}\otimes\mathrm{id}_F
$
for every $j\in \bigcup_{k=0}^{\dim X}J(\mathcal W^{(k)})$.\\
{\bf (I)} If $h_F$ is smooth, then the restriction map
\begin{align*}
\xyz{\Gamma(X,K_X\otimes S\otimes F)\ar[r]&
\Gamma\big(\underline{\mathcal W},(K_X\otimes S\otimes F)|_{\underline{\mathcal W}}\big)}
\end{align*}
is surjective.\\
{\bf (II)} If $(F,h_F)$ is a holomorphic line bundle with a singular hermitian metric, then the restriction map
\begin{align*}
\xyz{\Gamma\big(X,\mathcal I(h_F)(K_X\otimes S\otimes F)\big)\ar[r]&
\Gamma\big(\underline{\mathcal W},(K_X\otimes S\otimes F)|_{\underline{\mathcal W}}\big)}
\end{align*}
is surjective ($\mathcal I(h_F)$ being the multiplier ideal sheaf associated to $h_F$). 
\end{thmm}
The following is an immediate consequence of Theorem \ref{from snc}.
\begin{coro}
Suppose that $F$ is a holomorphic vector bundle on $X$ such that $\mathcal O_{\mathbf P(F)}(1)$ admits a smooth hermitian metric with strictly positive curvature.\footnote{For example, $X$ is projective and $F$ is an ample vector bundle in the sense of Hartshorne.} For an snc family $\mathcal W$, the restriction map
\begin{align*}
\xyz{\Gamma\big(X,K_X\otimes S\otimes \det F\otimes \mathrm{Sym}^k F\big)\ar[d]\\
\Gamma\big(\underline{\mathcal W},K_X\otimes S\otimes \det F\otimes \mathrm{Sym}^k F|_{\underline{\mathcal W}}\big)}
\end{align*}
is surjective. 
\end{coro}
Note that in Theorem \ref{from snc} we assume $\mathcal W$ to be $\mathcal S$-snc rather than assume the divisor $S=S_1+\cdots+S_q$ to be snc. In the special case that $S=S_1+\cdots+S_q$ itself is snc and $\underline{\mathcal W}$ is the (scheme-theoretic) complete intersection $S_1\cap\cdots\cap S_k$ for some $k\leqslant q$, if we further assume $F$ to be an ample line bundle, the surjectivity in {\bf(I)} can be obtained as follows: we consider the Koszul resolution
$$
\xyz{0\ar[r]& M_N\ar[r]&\cdots\ar[r]&M_1\ar[r]&\mathcal O_X\ar[r]&\mathcal O_{\underline{\mathcal W}}\ar[r]&0}
$$
where
$$
M_l = \oplus_{J\subseteq\{1,\dots,k\},\, |J|=l}\mathcal O_X(-S_J). 
$$
Applying the Kawamata--Viehweg vanishing theorem, we have for $l>0$ that
$$
H^l\big(X, M_l\otimes \mathcal O_X(K_X\otimes S\otimes F)\big)\simeq  \oplus_{J\subseteq\{1,\dots,k\},\, |J|=l} H^l\big(X, \mathcal O_X(K_X\otimes S^J\otimes F)\big) = 0, 
$$
and simple diagram chasing yields the desired surjectivity. However, the above argument cannot be applied for a general $\mathcal S$-snc family $\mathcal W$ as in Theorem \ref{from snc}, even if $F$ is assumed ample. On the other hand, Cao, Demailly, and Matsumura \cite{cdm} have obtained a very general qualitative extension theorem for an adjoint line bundle $K_X\otimes E$ equipped with a singular metric $h$ together with some quasi-psh function $\psi$ with analytic singularities satisfying certain curvature conditions. However, the singular metric case of Theorem \ref{from snc} is not an immediate consequence of their result, in the sense that \cite{cdm} extends sections from the analytic subspace $V\big(\mathcal I(he^{-\psi})\big)$, which does not seem to have an explicit description when $h$ is a less explicit singular metric, even if $V\big(\mathcal I(\psi)\big)$ defines a simple normal crossing divisor; in our case the center for extension is rather explicit, namely, the underlying space associated to an admissible family. The proof of Theorem \ref{from snc} will be given in Section \ref{app} based on a strategy of successive extension-correction suggested by Demailly, the idea of which appeared already in \cite{dema16}, together with a qualitative extension theorem (Theorem \ref{ext snc}). Validity of the strong openness conjecture \cite{gz2} (especially the slightly stronger forms in \cite{hiep} and \cite{lempert}) is used in the proof of the singular metric case of Theorem \ref{from snc} (cf.\ Lemma \ref{van ext}), and this is the place the triviality of multiplier sheaves plays a role. In view of Theorem \ref{from snc}{\bf (II)}, it seems natural to ask the following
\begin{quest}
Suppose that $\mathcal W$ is an $\mathcal S$-snc family and $(F,h_F)$ a holomorphic line bundle on $X$ with a singular hermitian metric such that 
$
\sqrt{-1}\, \Theta_{h_F}\succcurlyeq\pm\sqrt{-1}\, \Theta_{h_j}$ for every $j\in \bigcup\nolimits_{k=0}^{\dim X}J(\mathcal W^{(k)})
$. 
Assume that $h_F|_{W\cap S_J}$ is well-defined for every $W$ and for every $J$. For every $u\in\Gamma\big(S,\mathcal O_S(K_X\otimes S\otimes F)\big)$ such that 
$
u_W|_{W\cap S_J}\in \Gamma\big(W,\mathcal I(h_F|_W)(K_X\otimes S\otimes F)\big)
$
for every $W\in\bigcup\nolimits_{k=0}^{\dim X}\mathcal W^{(k)}$ and every $J\subseteq\{1,\dots,q\}$, does $u$ admit an extension $U$ in $\in\Gamma\big(X,\mathcal I(h_F)(K_X\otimes S\otimes F)\big)$?
\end{quest}
Note that if we drop the well-definedness requirement, i.\!\ e., allow $h_F|_{W\cap S_J}=+\infty$ and define $\mathcal I(h_F|_{W\cap S_J})$ to be $0$ if this is the case, then a negative answer has been given by a concrete example in \cite{ot05}.

Now we introduce the main technical ingredient (Theorem \ref{ext snc}), which is an extension theorem of the Ohsawa--Takegoshi type.  
\begin{defi}\label{GA}\footnote{$\mathcal G_A$ here essentially consists of all functions $G$ with  $G''(\cdot)=c_A(\cdot)e^{-(\cdot)}$ fulfilling the conditions in Theorem 2.1 of \cite{gz}.} For $A\in (-\infty,\infty]$, we define $\mathcal G_A$ to be the set of functions $G\in C^\infty((-A,\infty))$ such that 
$$
\quad\quad G^{(j)}(-A^+):=\lim\limits_{y\rightarrow -A^+}G^{(j)}(y)\ (j=0,1,2)\text{ all exist and are strictly positive,}
$$
$$
\quad\quad G''>0,\ G'^2-GG''>0,\ \text{and }
$$
\begin{align}\label{G limit infty}
G'(\infty):=\lim\limits_{y\rightarrow \infty}G'(y)\ \text{ exists and is strictly positive.}
\end{align}
\end{defi}
It is clear that $G>0$ and $G'>0$ for every $G\in\mathcal G_A$.

\begin{thmm}\label{ext snc} Let $\big((F,h_F),\mathcal W, \delta, \delta_\bullet, \gamma_\bullet, A, G\big)$ 
consist of
\begin{itemize}
\item[(i)] a holomorphic vector bundle $F$ on $X$ with a 
measurable metric $h_F$,  
\item[(ii)] an $\mathcal S$-admissible family $\mathcal W$,
\item[(iii)] numbers $\gamma_j\geqslant 0$, $\delta
> 0$, $0\leqslant\delta_j<1$, and $\varepsilon>0$ such that 
$$
\gamma_j > 0\ \text{ and }\ \delta_j=\delta\ \text{ for every }j\in J(\mathcal W),
$$ 
\item[(iv)] $A\in (-\infty, \infty]$ such that $|\sigma_S|_{h_S}^{2\gamma_\bullet} < e^A$,
\item[(v)] $G\in \mathcal G_A$. 
\end{itemize}
{\bf (I)} Assume that $h_F$ is smooth. If on $X\setminus \bigcup_{j=1}^q S_j$ we have
\begin{align}\label{cur snc 1}
\widetilde\Theta
:=
\sqrt{-1}
\Bigg(\Theta_{h_F}+\bigg(\sum_{j=1}^q\delta_j\Theta_{h_j}
+
\delta\sum\limits_{W\in\mathcal W}
|J_W|\,
\partial\overline\partial\log\sum\limits_{j\in J_W}|\sigma_j|_{h_j}^2\bigg)\otimes\mathrm{id}_F\Bigg)\geqslant 0
\end{align}
and
\begin{align}\label{cur snc 2}
\sqrt{-1}\left(\widetilde\Theta-
\frac{G'\big(-\log|\sigma_S|_{h_S}^{2\gamma_\bullet}\big)}{G\big(-\log|\sigma_S|_{h_S}^{2\gamma_\bullet}\big)}\sum\limits_{j=1}^q \gamma_j\Theta_{h_j}\otimes\mathrm{id}_F\right)\geqslant 0
\end{align}
in the sense of Nakano, and if
\begin{align}\label{bdd below}
G''\big(-\log|\sigma_S|_{h_S}^{2\gamma_\bullet}\big)\,
\frac{|\sigma_S|_{h_S}^{2\delta_\bullet-2}}{\prod\limits_{J\in\mathrm{J}_0(\mathcal W)} \Big(\sum\limits_{j\in  J} |\sigma_j|_{h_j}^2\Big)^{\delta |J|}}\,\ \text{is locally bounded from below,}
\end{align}
then for any collection of sections 
$$
u_W\in \Gamma\big(W_{\mathrm{reg}}, K_{W_{\mathrm{reg}}}\otimes S^W\otimes F|_W)\big)\quad (W\in\mathcal W)
$$
such that  
\begin{align}\label{snc int finite}
\int_{W_{\mathrm{reg}}} 
\frac{|\sigma^{J_W}|_{h^{J_W}}^{2\delta_\bullet}}{
\prod\limits_{J'\in\mathrm{J}_0(\mathcal W)\setminus\{J_W\}}\,\Big(\sum \limits_{j\in J'}|\sigma_j|^2_{h_j}\Big)^{\delta
\,|J'|}}\left\langle \frac{u_W}{\sigma^{J_W}}\right\ra^2_{h_F}
<\infty\quad (W\in\mathcal W),
\end{align}
there exists $U\in \Gamma(X,K_X\otimes S\otimes F)$ such that $U_W=u_W$ for every $W\in\mathcal W$ and
\begin{align}\label{snc L2 est}
\begin{split}
&\int_X G''\big(-\log|s_S|_{h_S}^{2\gamma_\bullet}\big)\,
\frac{|s_S|_{h_S}^{2\delta_\bullet}}{\prod\limits_{J\in\mathrm{J}_0(\mathcal W)} \Big(\sum\limits_{j\in  J} |s_j|_{h_j}^2\Big)^{\delta |J|}} 
\left \langle \frac{U}{s_S}\right\ra^2_{h_F} 
\\
&\quad\leqslant 
G'(\infty)
\sum\limits_{W\in\mathcal W}\frac{B_{|J_W|-1}(\delta)}{\sum\limits_{j\in J_W}\gamma_j}
\int_{W_{\mathrm{reg}}} 
\frac{|\sigma^{J_W}|_{h^{J_W}}^{2\delta_\bullet}}{
\prod\limits_{J'\in\mathrm{J}_0(\mathcal W)\setminus\{J_W\}}\,\Big(\sum \limits_{j\in J'}|s_j|^2_{h_j}\Big)^{\delta
\,|J'|}}\left\langle \frac{u_W}{\sigma^{J_W}}\right\ra^2_{h_F}
\end{split}
\end{align}
with
$$
B_{k-1}(\delta)=\pi^k\int_{\Delta_{k-1}}
\frac{d\sigma_1\cdots d\sigma_{k-1}}{\big(\sigma_1\cdots\sigma_{k-1}(1-\sigma_1-\cdots-\sigma_{k-1})\big)^{(1-\delta)}}\quad(B_0(\delta):=1)
$$ 
and
$$
\Delta_{k-1}=\big\{(\sigma_1,\dots,\sigma_{k-1})\in [0,\infty)^k\,\big|\, \sigma_1+\cdots+\sigma_{k-1}\leqslant 1\big\}.
$$
{\bf (II)} The statement in {\bf(I)} still holds if we assume instead that $(F,h_F)$ is a line bundle with a measurable hermitian metric of $L^1_{\mathrm{loc}}$ weight, and replace (\ref{cur snc 1}) and (\ref{cur snc 2}) by the following family of conditions:
\begin{align}\label{cur snc'}
\begin{split}
&\Theta_{h_F}
-
\sum_{j\in J(\mathcal W)}
\bigg(
\begin{array}{c}
\delta(t_j-1)\\
+\frac{G'(-A^+)}{G(-A^+)}s\gamma_j
\end{array}
\bigg)
\,\Theta_{h_j}
-
\sum_{j\notin J(\mathcal W)}
\bigg(
\begin{array}{c}
\frac{G'(-A^+)}{G(-A^+)}s\gamma_j\\
-\delta_j
\end{array}
\bigg)
\,\Theta_{h_j}
\geqslant 0
\\
&\text{ for all }\ s\in [0,1]\ \text{ and }\ (t_j)\in \prod\limits_{j\in J(\mathcal W)} [0, T_j]\quad \Big( T_j:=\sum\limits_{W\,:\,j\in J_W}|J_W|\Big).
\end{split}
\end{align}
\end{thmm}

The proof of Theorem \ref{ext snc} occupies both Sections \ref{general scheme} and \ref{ext from strata}. Many proofs of theorems of Ohsawa--Takegoshi type can be separated into two parts, as we explain now. Roughly speaking, given a holomorphic section $u$ on a ``center'' $Z$, the first part of the argument modifies a rather arbitrary extension $U_0$, the $L^2$-norm of which one does not have estimate of, to obtain an extension $F_t$ with $L^2$-norm bounded by a quantity $I_t$ with a parameter $t$. $I_t$ usually appears as the $L^2$-norm of $U_0$ with respect to some ``density'' $\chi_t$ supported in a neighborhood $V_t$ of $Z$ on which $u$ is defined; when $t\rightarrow-\infty$, the ``mass'' associated to the $\chi_t$ concentrates to the center. The second part of the argument rests on showing that $\liminf\limits_{t\rightarrow-\infty} I_t$ is bounded by the $L^2$-norm of $u$ over $Z$ with respect to a suitable measure. In Section \ref{general scheme} we formalize the first part of argument into a metatheorem (Theorem \ref{meta ext}), the proof of which follows very closely that of Theorem 2.1 in \cite{gz}. A difference between our treatment and that of \cite{gz} is that we introduce not only one but two ``weight functions'' $\psi$ and $\phi$. They usually appear as the same function in previous work of other authors (e.\,g., $\Psi$ in \cite{gz}), but allowing them to be different makes the curvature conditions in our extension theorem more flexible. As for the second part of argument, an upper bound of $\liminf\limits_{t\rightarrow-\infty} I_t$ can be given by the $L^2$-norm of $u$ with respect to an abstractly defined measure due to Ohsawa \cite{oh01} (see also \cite{gz} and \cite{dema16}). However, even when $Z$ is a simple normal crossing divisor, Ohsawa's measure for a general singular metric does not seem to have explicit expressions,\footnote{However, the author was informed by Dano Kim that a more detailed understanding of the Ohsawa measure does exist according to his unpublished work.} which is often necessary in practical applications. When the centers of extension bear sufficient regularity one may obtain a more explicit upper bound. We give in Section \ref{ext from strata} an explicit upper bound of $\liminf\limits_{t\rightarrow-\infty} I_t$ by direct computations. Roughly speaking, the function $\psi$ defines the center $Z$, and $\phi$ determines the way the aforementioned neighborhood $V_t$ shrinks. In previous studies where explicit upper bounds are obtained, one mainly considers the situation $\phi=\psi=\log|\sigma_S|_{h_S}^2$ with $\sigma_S$ a holomorphic section of some vector bundle $E$ such that $\wedge^{\mathrm{rk}\, E} d\sigma_S$ is essentially nonzero along its zero set $Z=Z(\sigma_S)$; if this is the case, $V_t$ shrinks to $Z$ as a usual tubular neighborhood shrinks when its radius tends to zero, and the computation is rather simple. In our proof of Lemma \ref{int estimate} both $\phi$ and $\psi$ are less regular and hence complicate the calculation, which was achieved by using suitably chosen local coordinate systems.
  
\subsection*{\bf Acknowledgement} The main part of this paper is written during the period I visited Institut Fourier in 2018-2019. I am most grateful to Professor Demailly for his warm hospitality and generosity in sharing many of his interesting and challenging ideas and in sparing plenty of time for discussions. The strategy of proving Theorem \ref{from snc} was essentially suggested by him, and some of the crucial settings of Theorem \ref{ext snc} are inspired by discussions with him. I am indebted to Professors Qi'an Guan and Xiangyu Zhou for patiently answering my several questions about their work \cite{gz}. I also like to thank Mario Chan, Jiun-Cheng Chen, Shin-Yao Jow, Yih Sung, and Sz-Sheng Wang for many stimulating discussions, and thank Professors Dano Kim, and Takeo Ohsawa for giving comments. Finally, I would like to thank Professors Shing-Tung Yau, Yng-Ing Lee, Hui-Wen Lin and Chin-Lung Wang for their constant supports and encouragements.

\section{{\bf A general scheme for entension of holomorphic sections}}\label{general scheme}

In this section we provide the first part of the proof of Theorem \ref{ext snc} in a form of a metatheorem, Theorem \ref{meta ext}. Although our argument models on that for Theorem 2.1 in \cite{gz}, we include a proof here for completeness since some adjustments are made as explained in the introduction.

The following setting will be in force throughout the current section. 
\begin{set}
Let $\big(Y, (E,h), Z, Z_0,\psi,\phi, A,G\big)$ consist of
\begin{itemize}
\item[(i)] a K\"ahler manifold $Y$, 
\item[(ii)] an analytic subspace $Z=V(\mathcal I_Z)$ of $Y$,
\item[(iii)] an analytic set $Z_0$ of $Y$,
\item[(iv)] a holomorphic vector bundle $E$ on $Y$ with a smooth hermitian metric $h$,
\item[(v)] two usc functions $\xyz{Y\ar[r]^-{\psi}&[-\infty, \infty)}$ and $\xyz{Y\ar[r]^-{\phi}&[-\infty, A)}$ with $A\in (-\infty,\infty]$, and
\item[(vi)] a strictly positive function $G\in C^\infty\big((-A,\infty)\big)$
\end{itemize}
such that
\begin{align}\label{phi Z0}
\phi|_{Y\setminus Z_0}\text{ is smooth and }\phi^{-1}(-\infty)=Z_0,
\end{align}
\begin{align}\label{psi Z}
\psi|_{Y\setminus Z_0}\text{ is smooth and }\mathcal I(\psi)\subseteq\mathcal I_Z,
\end{align}
\begin{align}\label{cur 1}
\sqrt{-1}\Big(\Theta_{h}+\partial\overline\partial\psi\otimes\mathrm{id}_E\Big)\geqslant 0 \text{ in the sense of Nakano on } Y\setminus Z_0,
\end{align}
and 
\begin{align}\label{cur 2}
\sqrt{-1}\left(\Theta_{h}+\Big(\partial\overline\partial\psi+\frac{G'(-\phi)}{G(-\phi)}\partial\overline\partial\phi\Big)\otimes\mathrm{id}_E\right)\geqslant_{\mathrm{Nak}} 0\text{ on }Y\setminus Z_0.
\end{align}
\end{set}

Before giving the precise statement of the theorem, we first introduce a family of functions which will be used to construct truncations when dealing with singularity. 


\begin{defi}\label{Xab}
For every $0<a<b$ we let 
$$
\mathcal X_{a,b}:=\left\{\chi\in C^\infty(\mathbf R)\,\left|\,\chi\geqslant 0,\ \mathrm{supp}\ \chi\subseteq (a, b),\,\text{ and }\,\int_\mathbf R\chi=1\right.\right\}.
$$ 
For any $\chi\in\mathcal X_{a,b}$ we define $\tau_\chi\in C^\infty(\mathbf R)$ by integrating $\chi$ twice with initial values $\tau_\chi'(b)=1$ and $\tau_\chi(b)=b$.
\end{defi}
We remark that $\tau_\chi$ is convex and remains constant on $(-\infty,a]$, and $\tau_\chi(y)=y$ for every $y\in [b,\infty)$; in particular, $\tau_\chi\geqslant a$ and $0\leqslant \tau_\chi'\leqslant 1$. 
 
We state the following metatheorem. 
\begin{thm}\label{meta ext} Given $(u,\Omega,U_0,G)$ consisting of
\begin{itemize}
\item[(i)] a section $u\in\Gamma\big(Z,\mathcal O_Y(K_Y\otimes E)|_Z\big)$,
\item[(ii)] a relatively compact open subset $\Omega$ of $Y$ admitting a complete K\"ahler metric,  
\item[(iii)] a section $U_0\in\Gamma\big(Y,\mathcal O_Y(K_Y\otimes E)\big)$ such that $U_0|_{Z\cap \Omega}=u|_{Z\cap \Omega}$, and
\item[(iv)] $G\in \mathcal G_A$,
\end{itemize}
if $\chi\in\mathcal X_{a,b}$ for some $a<b$ and 
$$
\liminf\limits_{t\rightarrow-\infty} \int_{\Omega} \chi(\phi -t)\,e^{-\psi}\langle U_0\ra^2_{h} <\infty,
$$
then there exists $U_\Omega\in\Gamma\big(\Omega,\mathcal O_Y(K_Y\otimes E)\big)$ such that 
$U_{\Omega}|_{Z\cap \Omega}=u|_{Z\cap \Omega}$ and
\begin{align}\label{general scheme estimate}
\int_\Omega G''(-\phi)\,e^{-\psi}\langle U_\Omega\ra^2_{h}
\leqslant 
G'(\infty)\, \liminf\limits_{t\rightarrow-\infty} \int_{\Omega} \chi(\phi -t)\,e^{-\psi}\langle U_0\ra^2_{h}.
\end{align}
\end{thm}
The proof of this theorem will be reduced to the following result. 
\begin{lem}\label{main lem}
Let $(u,G,\Omega,U_0)$ be as in Theorem \ref{meta ext}, and $a<b$ be real numbers such that $(G/G')(-b)\geqslant b-a$.\footnote{This holds when $b\rightarrow-\infty$ with $b-a$ bounded by a fixed constant. By (\ref{G limit infty}) we have
$
\lim\limits_{y\rightarrow\infty}\frac{G(y)}{y}=\lim\limits_{y\rightarrow\infty}G'(y)=G'(\infty)>0
$, and hence 
$
(G'/G)(y)\rightarrow 0
$
as $y\rightarrow\infty$.
}
For every $\chi\in \mathcal X_{a,b}$ there exists a section $U_\chi\in\Gamma\big(\Omega,\mathcal O_Y(K_Y\otimes E)\big)$ such that 
$U_\chi|_{\Omega\cap Z}=u|_{\Omega\cap Z}$ and 
$$
\int_{\Omega} G''\big(-\tau_\chi(\phi)\big)\,e^{-\psi}\big\langle U_\chi-\big(1-\tau_\chi'(\phi)\big)U_0\big\ra^2_{h}
\leqslant 
G'(\infty)\int_{\Omega}\chi(\phi)\, e^{-\psi}\langle U_0\ra_{h}^2.
$$
\end{lem}
\begin{proof}[Proof of Theorem \ref{meta ext} assuming Lemma \ref{main lem}]
Select a sequence $t_j\rightarrow -\infty$ as $j\rightarrow\infty$ such that 
$$
\lim\limits_{j\rightarrow\infty} \int_{\Omega}\chi(\phi-t_j)\,e^{-\psi}\langle U_0\ra^2_{h}=\liminf\limits_{t\rightarrow-\infty} \int_{\Omega} \chi(\phi -t)\,e^{-\psi}\langle U_0\ra^2_{h}.
$$
Let $a_j:=a+t_j$ and $b_j:=b+t_j$. We have $(G/G')(-b_j)\geqslant b_j-a_j=b-a$ for $j$ sufficiently large. Applying Lemma \ref{main lem} with the functions
$
\chi_j(\cdot):=\chi(\cdot -t_j)\in\mathcal X_{a_j,b_j}
$
we obtain sections
$$
U_{\chi_j} \in\Gamma\big(\Omega,\mathcal O_Y(K_Y\otimes E)\big)
$$
such that
$$
\int_{\Omega} G''\big(-\tau_{\chi_j}(\phi)\big)\,e^{-\psi}\big\langle U_{\chi_j}-\big(1-\tau_{\chi_j}'(\phi)\big)U_0\big\ra^2_{h}
\leqslant 
G'(\infty)\int_{\Omega}\chi_j(\phi)\, e^{-\psi}\langle U_0\ra_{h}^2.
$$
Fix a sequence of relatively compact open sets 
$$O_1\subseteq O_2\subseteq\cdots\subseteq O_m\subseteq O_{m+1}\subseteq\cdots$$ 
of $Y$ such that $Y\setminus Z=\bigcup\limits_m O_m$. There exists a strictly increasing sequence $j_m$ such that $b+t_{j_m}\leqslant \min\limits_{\overline O_m}\phi$ for every $m\in\mathbf N$. Then $\tau_{\chi_{j_m}}(\phi)|_{O_m}=\phi|_{O_m}$ and $\tau_{\chi_{j_m}}'(\phi)|_{O_m}=1$, and hence 
\begin{align*}
\begin{split}
&\int_{O_m} G''\big(-\tau_{\chi_{j_m}}(\phi)\big)\,e^{-\psi}\big\langle U_{\chi_{j_m}}\big\ra^2_{h}
\leqslant 
G'(\infty)\int_{\Omega}\chi_{j_m}(\phi)\,e^{-\psi} \langle U_0\ra_{h}^2\quad (m\in\mathbf N).
\end{split}
\end{align*}
To get the desired section, we seek to apply Lemma \ref{loc L2 conv => loc unif conv} with 
$$
(M,U_m,w_m)=\Big(\Omega,U_{\chi_{j_m}},G''\big(\!-\tau_{\chi_{j_m}}(\phi)\big)\Big).
$$
Note that $w_m$ converges to $w=G''(-\phi)$ almost everywhere as $m\rightarrow\infty$, and it is direct to check that the conditions of Lemma \ref{loc L2 conv => loc unif conv} are fulfilled to yield a section $U_\Omega\in\Gamma\big(\Omega,\mathcal O_Y(K_Y\otimes E)\big)$ such that 
$U_\Omega|_{Z\cap \Omega}=u|_{Z\cap \Omega}$ and
$$
\int_\Omega G''(-\phi)\,e^{-\psi}\langle U_\Omega\ra^2_{h}
\leqslant 
G'(\infty)\, \liminf\limits_{m\rightarrow\infty} \int_{\Omega} \chi_{j_m}(\phi)\,e^{-\psi}\langle U_0\ra^2_{h}.
$$
\end{proof}
Now it remains to prove Lemma \ref{main lem}, and the rest of this section constitutes a proof.

\begin{pt}[{\bf $\overline\partial$-equations with truncated data}] 
For simplicity, in the following discussions we denote $\tau_\chi$ by $\tau$ for $\chi\in\mathcal X_{a,b}$. We manually extend the domains of $\tau$, $\tau'$, and $\tau''=\chi$ to include $-\infty$ by defining 
\begin{align}\label{value at infty}
\tau^{(j)}(-\infty):=\lim\limits_{y\rightarrow -\infty}\tau^{(j)}(y)=\tau^{(j)}(a)\quad (j=0,1,2).
\end{align}
Then $\tau(\phi)$, $\tau'(\phi)$, and $\tau''(\phi)$ are all smooth functions on $Y$ whose restrictions to the open neighborhood $\phi^{-1}\big([-\infty, a)\big)$ of $Z_0$ are constant functions. In particular,
\begin{align}\label{rhs}
\eta_\chi:=\overline\partial\Big(\big(1-\tau'(\phi)\big)U_0\Big),
\end{align}
can be viewed as a {\it smooth} $E$-valued $(n,1)$-form on $Y$. We also note that $\tau'(\phi)$, $\tau''(\phi)$, and $\eta_\chi$ all vanish on $\phi^{-1}\big([-\infty, a)\big)$.
We will first obtain a solution $\gamma=\gamma_{\chi}$ to the $\overline\partial$-equation
\begin{align}\label{d-bar to slove}
\overline\partial \gamma = \eta
\end{align}
on $\Omega\setminus Z$ with $L^2$-estimate and define  
$$
U_{\chi}:=\big(1-\tau'(\phi)\big)U_0 - \gamma_{\chi}.
$$
Then we shall obtain by a limiting process a holomorphic extension $U_\Omega$ of $u|_{Z\cap\Omega}$ on $\Omega$ with $L^2$-estimate.  
\end{pt}

\begin{pt}[{\bf Choosing auxiliary weight functions by solving ODEs}]\label{ode}
We first explain the presence of part of the conditions in Definition \ref{GA}, following the ODE approach in \cite{gz}. To exploit Theorem \ref{L2 est tool} to solve the aforementioned $\overline\partial$-equation with $L^2$-estimate, for any $\chi\in\mathcal X_{a,b}$ we consider functions $\lambda$, $\mu$, and $\nu$ of the form (with $\tau_\chi$ denoted by $\tau$)
$$\lambda=s(-\tau\circ\phi),\, \mu=t(-\tau\circ\phi),\ \text{ and }\ \nu=w(-\tau\circ\phi)$$
where $s>0$, $t>0$, and $w$ are smooth functions on $(-A,\infty)$, and define
\begin{align}\label{h0,h,}
\widehat h(\cdot,\cdot):=h(\cdot,\cdot)e^{-\psi}\ \text { and }\ \widetilde h(\cdot,\cdot)=h(\cdot,\cdot)e^{-\psi}e^{-\nu},
\end{align}
whose restriction on $Y\setminus Z_0$ is then smooth. 
To check the validity of (\ref{appriori for twisted dbar}), we proceed a direct calculation and obtain\footnote{Hereinafter we omit $\otimes\mathrm{id}_E$ from all appearances of $(1,1)$-forms for simplicity.}
\begin{align*}
\begin{split}
&\lambda  \Theta_{\widetilde h} - \partial \overline{\partial}
\lambda - \frac{1}{\mu}\partial{\lambda}\wedge\overline\partial{\lambda}
=
\lambda \Theta_{\widehat h} + \lambda\partial \overline{\partial}\nu-\partial \overline{\partial}
\lambda - \frac{1}{\mu}\partial{\lambda}\wedge\overline\partial{\lambda}\\
&\quad=(s'-sw')\circ(-\tau\circ\phi)\cdot (\tau''\circ\phi)\, \partial\phi\wedge\overline\partial\phi\\
&\quad\quad\quad +s(-\tau\circ\phi)\,\Theta_{\widehat h}+(s'-sw')\circ(-\tau\circ\phi)\cdot (\tau'\circ\phi)\, \partial\overline\partial\phi\\
&\quad\quad\quad+\left(sw''-s''-\frac{s'^2}{t}\right)\circ(-\tau\circ\phi)\cdot \partial(\tau\circ\phi)\wedge \overline\partial (\tau\circ\phi).
\end{split}
\end{align*}
As in \cite{gz}, we impose the following conditions on $s$, $t$, and $w$:
\begin{align}\label{s,t,u}
s>0,\ t>0,\ s'-sw'=1\ \text{ and }\  sw''-s''-\frac{s'^2}{t}=0.
\end{align}
Then $(e^{-w}s)'=e^{-w}$, and hence $s$ is of the form $e^w\int e^{-w}$. By setting 
$G=\int e^{-w}$, we have $G'=e^{-w}> 0$, 
\begin{align}\label{s,u}
s=\frac{G}{G'},\ \text{ and }\ w=-\log G'.
\end{align}
The second condition in (\ref{s,t,u}) implies that $s''-sw''-s'w'=(s'-sw')'=0$, and hence
\begin{align}\label{t}
t=\frac{-s'^2}{s''-sw''}=s\frac{-s'/s}{w'}.
\end{align} 
A direct computation yields that
$$
(s'/s)w'=\left(\frac{G'}{G}-\frac{G''}{G'}\right)\left(-\frac{G''}{G'}\right)=-\frac{(G'^2-GG'')G''}{GG'^2}
$$
and
\begin{align*}
s+t=s\left(1-\frac{s'/s}{w'}\right)=s\left(\frac{w'-(s'/s)}{w'}\right)=s\left(\frac{-1/s}{w'}\right)=\frac{-1}{w'}=\frac{G'}{G''}.
\end{align*}
Therefore, that $G=\int e^{-w}$ and the strictly positivity of $s$ and $t$ imply that 
\begin{align}\label{G conditions}
\ G>0,\ G'>0,\ G''>0,\ \text{ and }\ G'^2-GG''>0,
\end{align}
It is easy to see that, conversely, if $s$, $t$, and $w$ are determined by (\ref{s,u}) and (\ref{t}) according to a smooth function $G$ on $(-A,\infty)$ satisfying (\ref{G conditions}),
then (\ref{s,t,u}) holds. Note that every $G\in\mathcal G_A$ fulfils (\ref{G conditions}). We note for later use that
\begin{align}\label{G', G'' as wt}
e^{-w}=G'\ \text{ and }\ \frac{e^{-w}}{s+t}=G''. 
\end{align}
\end{pt}
\begin{pt}[{\bf Solving the $\overline\partial$-equations}] 
Now we assume that $G\in\mathcal G_A$ and that the functions $t$ and $u$ are constructed as in (\ref{s,u}) and (\ref{t}).
\begin{ptlem}\label{key} 
If $b-a\leqslant s(-b)$, then 
\begin{align*}
\sqrt{-1}\left(\lambda \Theta_{\widetilde h} - \partial \overline{\partial}
\lambda - \frac{1}{\mu}\partial{\lambda}\wedge\overline\partial{\lambda}\right)
\geqslant_{\mathrm{Nak}}
\sqrt{-1}\,\chi(\phi)\, \partial\phi\wedge\overline\partial\phi
\end{align*}
for every $\chi\in\mathcal X_{a,b}$.
\end{ptlem}
\begin{proof}
We have
\begin{align*}
\begin{split}
&\quad\quad\sqrt{-1}\left(\lambda  \Theta_{\widetilde h} - \partial \overline{\partial}
\lambda - \frac{1}{\mu}\partial{\lambda}\wedge\overline\partial{\lambda}\right)
\\
&\quad\quad\quad=
\sqrt{-1}\Big((\chi\circ\phi)\,\partial\phi\wedge\overline\partial\phi
+
\Big[s(-\tau\circ\phi)\,\big(\Theta_h+\partial\overline\partial\psi\big)
+ 
(\tau'\circ\phi)\, \partial\overline\partial\phi\Big]\Big).
\end{split}
\end{align*}
It suffices to show that the last term is nonnegative if $b-a\leqslant s(-b)$. Note that $s=\frac{G}{G'}$. By applying Lemma \ref{stronger} below with $S=s$, we see that\footnote{This conclusion corresponds to (5.7) in \cite{gz}, which was stated there without a proof. It can be verified by Lemma \ref{stronger} here, which I learned via private communication with Guan and Zhou. The proof given here follows essentially their argument.} 
\begin{align*}
\begin{split}
& \sqrt{-1}\Big(s(-\tau\circ\phi)\big(\Theta_{h}+ \partial\overline\partial\psi\big)+ (\tau'\circ\phi)\, \partial\overline\partial\phi\Big)
\\
&
\quad\quad\quad\geqslant 
\sqrt{-1}(\tau'\circ\phi)\,\Big(s(-\phi)\big(\Theta_{h}+\partial\overline\partial\psi\big)+ \partial\overline\partial\phi\Big)\geqslant 0
\end{split}
\end{align*}
by our curvature assumption (\ref{cur 1}) and (\ref{cur 2}).
\end{proof}
\begin{ptlem}\label{stronger} For any $\chi\in\mathcal X_{a,b}$ and any $C^1$ function $S$, if 
$S'<1$ and $b-a\leqslant S(-b)$, 
then 
$
S\big(\!-\tau_\chi(y)\big)\geqslant S(-y)\,\tau_\chi'(y).
$ 
\end{ptlem}
\begin{proof} 
We only need to consider those $y$ lying in $(a, b)$ since $\tau'(y)=0$ if $y\leqslant a$, and $\tau(y)=y$ if $y\geqslant b$. For $y\in(a,b)$, then
\begin{align*}
\begin{split}
&S\big(\!-\tau(y)\big)-S(-y)\, \tau'(y)\\
&\quad\quad = S\big(\!-\tau(y)\big)-S(-y)+S(-y)-S(-y)\, \tau'(y)\\
&\quad\quad \geqslant -\tau(y)+y + S(-y)\big(1-\tau'(y)\big)\ \text{ by the mean value theorem and } S'<1\\
&\quad\quad =\int_y^b \big(\tau'(z)-1\big)dz+ S(-y)\big(1-\tau'(y)\big)\\
&\quad\quad \geqslant (b-y)\big(\tau'(y)-1\big) + S(-y)\big(1-\tau'(y)\big)\ \text{ since }\big(\tau'-1\big)'=\tau''\geqslant 0\\
&\quad\quad \geqslant \big(1-\tau'(y)\big)\big(S(-b)-b+a\big)\geqslant 0.
\end{split}
\end{align*}
\end{proof}
Recall the right hand side of (\ref{d-bar to slove}):
$$
\eta_\chi=\overline\partial\Big(\big(1-\tau'(\phi)\big)U_0\Big)=-\chi(\phi)\,\overline\partial\phi\wedge U_0
$$ 
To apply Theorem \ref{L2 est tool}, first note that $\Omega\setminus Z_0$ admits a complete K\"ahler metric (\cite{dema82} Lemma 1.5). Besides, since $\lambda$ and $\mu$ are smooth on $Y$, they are bounded on the relatively compact set $\Omega\setminus Z_0$. As for condition (\ref{appriori for twisted dbar}), for every $v\in\mathcal D^{n,1}(\Omega\!\setminus\! Z_0,\,E|_{\Omega\setminus Z_0})$ and (a fixed K\"ahler metric $g$) we have pointwise that 
\begin{align*}
\begin{split}
&(\eta_\chi,v)_{g,\widetilde h}
=-\big(\chi(\phi)\,\overline\partial\phi\wedge U_0, v\big)_{g,\widetilde h}
=-\big(\sqrt{\chi(\phi)}\, U_0, \sqrt{\chi(\phi)}\,\iota^{\overline\partial\phi}v\big)_{g,\widetilde h}
,
\end{split}
\end{align*}
and hence by Lemma \ref{key},
\begin{align*}
\begin{split}
\big|(\eta_\chi,v)_{L^2_{g,\widetilde h}}\big|^2 
&=
\left|\int_{\Omega\setminus Z_0}\big(\sqrt{\chi(\phi)}\, U_0, \sqrt{\chi(\phi)}\,\iota^{\overline\partial\phi}v\big)_{g,\widetilde h}\,dV_g\right|^2\\
&\leqslant
\int_{\Omega\setminus Z_0}\chi(\phi)\, |U_0|_{g,\widetilde h}^2\, dV_g
\int_{\Omega\setminus Z_0}\sqrt{-1}\chi(\phi)\, (\partial\phi\wedge\overline\partial\phi)[v,v]_{g,\widetilde h}\, dV_g\\
&\leqslant 
C'\int_{\Omega\setminus Z_0}\sqrt{-1}\Big(\lambda \Theta_{\widetilde h} - \partial \overline{\partial}
\lambda - \frac{1}{\mu}\partial{\lambda}\wedge\overline\partial{\lambda}\Big)[v,v]_{g,\widetilde h}\, dV_g
\end{split}
\end{align*}
where 
$$C'=\int_{\Omega\setminus Z_0}\chi(\phi)\, |U_0|_{g,\widetilde h}^2\, dV_g=\int_{\Omega\setminus Z_0}G'\big(-\tau_\chi(\phi)\big)\,e^{-\psi}\,\chi(\phi)\, \langle U_0\ra^2_h.$$
Since $w'=\frac{s'-1}{s}<0$ and $e^{-w}=G'$, we have $C'\leqslant 
G'(\infty)\,I_{\chi}$
with
\begin{align}
I_{\chi}:=\int_{\Omega\setminus Z_0}\chi(\phi)\,e^{-\psi}\langle U_0\ra^2_h<\infty.
\end{align}
By Theorem \ref{L2 est tool}, there exists a measurable $E$-valued  $(n,0)$-form $\gamma_{\chi}$ such that
\begin{align}\label{d-bar gamma}
\quad\quad\overline\partial \gamma_{\chi} = \overline\partial\Big(\big(1-\tau'(\phi)\big)U_0\Big)\text{ on } \Omega\setminus Z_0\text{ in the sense of current}
\end{align}
and
\begin{align}\label{gamma I}  
\begin{split}
&\quad\int_{\Omega\setminus Z_0} G''\big(-\tau_\chi(\phi)\big)\, e^{-\psi}\,|\gamma_{\chi}|^2_{g, h}\, dV_g \leqslant
G'(-\infty)\, I_{\chi}.
\end{split}
\end{align}
We let
$$
U_{\chi}:=\big(1-\tau'(\phi)\big)U_0-\gamma_{\chi}.
$$ 
By (\ref{d-bar gamma}), $U_\chi$ is a holomorphic section of $K_Y\otimes E$ on $\Omega\setminus Z_0$. Since
$
G''\big(-\tau(\phi)\big)|_{\Omega}
$
has strictly positive lower bounds, (\ref{gamma I}) implies that 
\begin{align}\label{h0 h L2}
\int_{\Omega\setminus Z_0}  \langle\gamma_{\chi}\ra^2_h\leqslant
e^{\sup_{\Omega} \psi}
\int_{\Omega\setminus Z_0}  e^{-\psi}\,|\gamma_{\chi}|^2_{g, h}\, dV_g <\infty.
\end{align}  
We clearly have
$$
\int_{\Omega} \,\big\langle\big(1-\tau'(\phi)\big)U_0\big\ra^2_{h}<\infty,
$$
and hence 
$$
\int_{\Omega\setminus Z_0} \langle U_\chi\ra^2_{h}<\infty. 
$$
Since $h$ is smooth, $U_\chi$ extends to a holomorphic section of $K_Y\otimes E$ on $\Omega$. As noted right after (\ref{rhs}), $\tau'(\phi)$ vanishes on an open neighborhood of $Z_0$, say, $V$. Thus, the measurable section
$
\gamma_\chi|_{V\cap \Omega},
$
which coincides with the holomorphic section 
$
U_\chi|_{V\cap \Omega}-U_0|_{V\cap \Omega}
$, 
is holomorphic. Now the last inequality in (\ref{h0 h L2}) implies that the local expression of $\gamma_\chi|_{V\cap \Omega}$ lies in the ideal sheaf $\mathcal I(\psi)\subseteq \mathcal I_Z$. This shows that $U_\chi|_{\Omega\cap Z}=U_0|_{\Omega\cap Z}=u|_{\Omega\cap Z}$. Finally, the required $L^2$-estimate is simply (\ref{gamma I}). This completes the proof of Lemma \ref{main lem}.
\end{pt}

\section{\bf Proof of Theorem \ref{ext snc}}\label{ext from strata}

In this section we finish the proof of Theorem \ref{ext snc}. The estimate (\ref{snc L2 est}) will be obtained by establishing an explicit upper bound of the right hand side of (\ref{general scheme estimate}). See Lemma \ref{int estimate} for the detail. 

We may restrict the given data to $X\setminus H$ for any analytic hypersurface $H$ which contains no $W\in\mathcal W$ and, by (\ref{bdd below}) and Riemann's extension theorem, it suffices to prove the theorem on $X\setminus H$. In particular, we may assume that $X$ is Stein, all vector bundles involved are trivial, and $S$ is a reduced snc divisor. Thus we have the adjunction maps
$$
\xyz{
\Gamma(X,K_X\otimes S\otimes F)
\ar[r]^-{(\cdot)_{S_J}}&
\Gamma\big(S_J,K_{S_J}\otimes (S^J\otimes F)|_{S_J}\big).}
$$

In the situation of {\bf (II)}, we may apply the standard regularization process and reduce to the ${\bf (I)}$ with $X$ a Stein manifold and $F$ a line bundle. Note that (\ref{cur snc'}) is preserved by the regularization process since $\widetilde\Theta_{s,t_\bullet}$ is a linear combination of $\Theta_{h_F}$ and $\Theta_j$ ($j=1,\dots,q$) with constant coefficients. Besides, when $h_F$ is smooth, (\ref{cur snc'}) implies both (\ref{cur snc 1}) and (\ref{cur snc 2}). This can be seen by Lemma \ref{cur comp} and the fact that $G'/G$ is strictly decreasing. Therefore, it remains to prove {\bf (I)} assuming that $X$ is Stein.

We may assume that, for every $J\subseteq\{1,\dots,q\}$, the family $\mathcal W$ contains either all or none of the irreducible components of $S_J$. More precisely, for every $J\subseteq \{1,\dots,q\}$ let $Q_J$ be the union of all irreducible components of $S_J$ which are not a member of $\mathcal W$. It suffices to remove from $X$ a hypersurface which contains $\bigcup_J Q_J$ but contains no $W\in\mathcal W$.

Let  
$
\mathrm J_0=\mathrm J_0(\mathcal W)\big(=\{J_W\,|\, W\in\mathcal W\}\big)
$. For every $J\in\mathrm J_0$, let 
$$
u_J\in \Gamma\big(S_J,K_{S_J}\otimes S^J\otimes F|_{S_J}\big)
$$ 
be the section determined by $u_J|_W=u_W$ for every $W$ associated to $J$, and fix an ``initial extension''  
$
U_0\in \Gamma(X,K_X\otimes S\otimes F)
$ 
such that 
$
(U_0)_{S_J}=u_J\ (J\in\mathrm J_0).
$
We fix an exhaustion 
$
\Omega_1\Subset\Omega_2\Subset\cdots
$
of $X$ with strongly pseudoconvex open subsets. It is clear that the proof will be completed by applying Lemma \ref{ext snc lem} below with $\Omega=\Omega_k$ for every $k$ together with a normal family argument. 

\begin{lem}\label{ext snc lem}  If $\Omega$ is a relatively compact open subset of $X$ admitting a complete K\"ahler metric, then there exists $U_\Omega\in\Gamma\big(\Omega,\mathcal O_X(K_X\otimes S\otimes F)\big)$ such that $(U_\Omega)_{S_J\cap \Omega}=u_J|_{S_J\cap \Omega}$ for every $J\in\mathrm J_0$ and
\begin{align*}
&\int_\Omega G''\big(-\log|\sigma_S|_{h_S}^{2\gamma_\bullet}\big)\,
\frac{|\sigma_S|_{h_S}^{2\delta_\bullet}}{\prod\limits_{J\in\mathrm J_0} \Big(\sum\limits_{j\in  J} |\sigma_j|_{h_j}^2\Big)^{\delta |J|}} 
\left \langle \frac{U}{\sigma_S}\right\ra^2_{h_F} 
\\
&\quad\leqslant 
G'(\infty)
\sum\limits_{J\in\mathrm J_0}\frac{B_{|J|}(\delta)}{\sum\limits_{j\in J}\gamma_j}
\int_{S_J} 
\frac{|\sigma^J|_{h^J}^{2\delta_\bullet}}{
\prod\nolimits_{J'\in\mathrm J_0\setminus\{J\}}\,\Big(\sum \limits_{j\in J'}|\sigma_j|^2_{h_j}\Big)^{\delta |J|}}
\left\langle \frac{u_J}{\sigma^J}\right\ra^2_{h_F}.
\end{align*}
\end{lem}

\begin{proof}
Applying Theorem \ref{meta ext} with 
$$
Y=X\setminus 
\Bigg(\bigcup\limits_{j\notin\bigcup\mathrm{J}_0} S_j
\cup
\Big(\bigcup\limits_{J\in\mathrm{J}_0} S_J\Big)_{\mathrm{sing}}\Bigg),
$$
$$
h=h_S\otimes h_F,\ Z=Y\cap\bigcup\limits_{j\in\bigcup\mathrm{J}_0}S_j,\ Z_0=Y\cap\bigcup\limits_{J\in\mathrm{J}_0} S_J,
$$
$$
\psi=\log\Bigg(|\sigma_S|^{2(1-\delta_\bullet)}_{h_S}\!
\prod\limits_{J\in\mathrm{J}_0}\,\Big(\sum \limits_{j\in J}|\sigma_j|^2_{h_j}\Big)^{\delta | J|}\Bigg), \text{ and }\, \phi=\log |\sigma_S|^{2\gamma_\bullet}_{h_S},
$$
we see that the proof will be completed by Lemma \ref{int estimate} below.
\end{proof}

\begin{lem} \label{int estimate}
\begin{align*}
\begin{split}
&\limsup\limits_{t\rightarrow-\infty} \int_{\Omega} \chi(\phi -t)\,e^{-\psi}\langle U_0\ra^2_{h_S\otimes h_F}\\
&\quad\quad 
\leqslant
\sum\limits_{J\in\mathrm{J}_0}\frac{B_{|J|-1}(\delta)}{\sum\limits_{j\in J}\gamma_j}
\int_{S_J} 
\frac{|\sigma^J|_{h_S^J}^{2\delta_\bullet}}{
\prod\limits_{J'\in\mathrm{J}_0\setminus\{J\}}\,\Big(\sum \limits_{j\in J'}|\sigma_j|^2_{h_j}\Big)^{\delta |J'|}}
\left\langle \frac{u_J}{\sigma^J}\right\ra^2_{h_F}.
\end{split}
\end{align*}
\end{lem}
\begin{proof}
It will be convenient to set $t=\log\varepsilon$ and define $\widetilde \chi(s) := \chi(\log s)$. We now analyse the integral
$$
I(\varepsilon)=\int_{\Omega}
\chi(\phi-\log\varepsilon)
\,e^{-\psi}\,\langle U_0\ra^2_{h_S\otimes h_F}
=
\int_{\Omega}\frac{\widetilde\chi(e^{\phi}/\varepsilon)\, \langle U_0\ra^2_{h_S\otimes h_F}}{|\sigma_S|^{2(1-\delta_\bullet)}_{h_S}\!\prod\limits_{J\in\mathrm{J}_0}\,\Big(\sum \limits_{j\in J}|\sigma_j|^2_{h_j}\Big)^{\delta | J|}}.
$$
We choose for every $p \in \overline\Omega$ a relatively compact  open neighborhood $V_p$ satisfying the following conditions: if $p\notin\underline{\mathcal W}$, then $\overline V_p\cap \underline{\mathcal W}=\emptyset$; if $p\in\underline{\mathcal W}$, then $\overline V_p\cap\underline{\mathcal W}=\overline V_p\cap S_{J_p}$. We let 
\begin{align}\label{Ap}
A_p(x)
:=
|\sigma^{J_p}(x)|_{h^{J_p}}^{2(1-\delta_\bullet)}
\cdot \prod\limits_{J\in\mathrm{J}_0\setminus\{J_p\}}\,\Big(\sum \limits_{j\in J}|\sigma_j(x)|^2_{h_j}\Big)^{\delta |J|}
\quad (x\in X)
\end{align}
and
$$
\widetilde A_p(x)
:=
\left(|\sigma^{J_p}(x)|_{h^{J_p}}^{2 \gamma_\bullet}\right)^{\frac{1}{\sum\nolimits_{j\in J_p}\gamma_j}}=
\Big(\prod\limits_{j\notin J_p}|\sigma_j(x)|^{2\gamma_j}_{h_j}\Big)^{\frac{1}{\sum\nolimits_{j\in J_p}\gamma_j}}
\quad (x\in X).
$$
If $J_p$ consists of $j_1<\dots<j_k$, there exist an open neighborhood $V_p$ of $p$, a coordinate system $(z_1,\dots,z_n)$ on $V_p$, and a frame $e_j$ of $S_j|_{V_p}$ for every $j\in\{1,\dots,q\}$ such that $\sigma_{j_i}|_{V_p}=z_ie_{j_i}$ ($i=1,\dots, k=| J_p|$) . 
We let $(x_i,y_i):=(\mathrm{Re}\!\ z_i,\mathrm{Im}\!\ z_i)$, 
and let
$\varphi_{j}\, (j=1,\dots,q)$ and $\varphi^{J_p}$
denote the local weights of $h_j\, (j=1,\dots,q)$ and $h^{J_p}$ with respect to the associated local frames, respectively.  
Note that both $A_p$ and $\widetilde A_p$ are strictly positive on $V_p$. We let $(\xi_i,\eta_i)=(\mathrm{Re}\, \zeta_i,\, \mathrm{Im}\,\zeta_i)$ with
$$
\zeta_i
:=
z_i\,e^{-\frac{1}{2}\varphi_{j_i}}\widetilde A_p^{\frac{1}{2}}
\quad (i=1,\dots,| J_p|).
$$
We will write $| J_p|$ as $k=k(p)$ from now on for simplicity. By shrinking $V_p$ suitably we may assume that the following statements hold.
\begin{itemize}
\item[•] 
$
(\xi_1, \eta_1,\dots, \xi_k,\eta_k,x_{k+1},y_{k+1},\dots,x_n,y_n)
$ 
is a real coordinate system on $V_p$ with respect to which
the open set $V_p$ is represented by the cartescian product $D_p\times V_p'$ with
$$
D_p=
\left\{(\xi_\bullet,\eta_\bullet)\in \mathbf R^{2k}\,\left|\,\sum\nolimits_{i=1}^k |\xi_i|^2+|\eta_i|^2\leqslant R^2\right.\right\}
$$ 
for some $R=R_p<1$ and $V_p'\subseteq\mathbf R^{2n-2k}\simeq\mathbf C^{n-k}$ an open set.
\item[•] The closure $\overline{V}_p$ is compact and is disjoint from $D_j$ if $j\notin J_p$; in particular, $A_p$ and $\widetilde A_p$ both have strictly positive lower bounds on $\overline V_p$.
\item[•] The coordinate systems $(z_1,\dots,z_n)$ and $
(\xi_1, \eta_1,\dots, \xi_k,\eta_k,x_{k+1},y_{k+1},\dots,x_n,y_n)
$ and the frame sections $e_1,\dots, e_q$ are all defined on a neighborhood of $\overline{V}_p$. 
\end{itemize}
With respect to such a coordinate system we have
\begin{align}\label{e psi}
e^{\psi}
=\left\{
\begin{array}{ccc}
|\zeta_1^2\cdots\zeta_k^2|^{1-\delta}
\big(|\zeta_1|^2+\cdots+|\zeta_k|^2\big)^{\delta k}A_p\widetilde A_p^{-k}&\text{if}&p\in\overline\Omega\cap\bigcup\nolimits_{J\in\mathrm{J}_0}S_J;\\
|\zeta_1^2\cdots\zeta_k^2|^{1-\delta}
A_p\widetilde A_p^{-(1-\delta)k}&\text{if}&p\in\overline\Omega\setminus\bigcup\nolimits_{J\in\mathrm{J}_0}S_J
\end{array}
\right.
\end{align}
and
\begin{align}\label{e phi}
e^{\phi}=|\zeta_1|^{2\alpha_1}\cdots |\zeta_k|^{2\alpha_k}\ \text{ with }\ (\alpha_1,\dots,\alpha_k):=(\gamma_{j_1},\dots,\gamma_{j_k}).
\end{align}
Choose a partition of unity $\{\varrho_l\}_{l\in\mathbf N}$ subordinate to $\{V_p\}_{p\in X}$. By the local finiteness of the partition of unity, we may assume that there is a finite set $P\subseteq \overline\Omega$ and $l(p)\in\mathbf N$ such that $(\sum_{p\in P}\varrho_{l(p)})|_{\overline\Omega}=1$ and $\mathrm{supp}\,\varrho_{l(p)}\subseteq V_p$ for every $p\in P$.  
Note that 
$
I(\varepsilon)\leqslant
\sum\limits_{p\in P} I_p(\varepsilon)
$
where
$$I_p(\varepsilon):=\int_{V_p} \varrho_{l(p)}\,\widetilde\chi(e^{\phi}/\varepsilon)\, e^{-\psi}\,\langle U_0\ra^2_{h_S\otimes h_F}.$$
We adopt the following abbreviations
$$\begin{array}{ll}
(x,y)=(x_1,y_1,\dots,x_n,y_n),& dx\, dy=dx_1\,dy_1 \cdots dx_n\, dy_n,\\
(\xi,\eta)=(\xi_1,\eta_1,\dots ,\xi_k,\eta_k),&  d\xi\, d\eta=d\xi_1\, d\eta_1\cdots d\xi_k\, d\eta_k,\\ 
(x',y')=(x_{k+1},y_{k+1},\dots ,x_n,y_n),& dx'\, dy'=dx_{k+1}\, dy_{k+1} \cdots dx_n\, dy_n
\end{array}$$
and write
$$U_0\big|_{V_p}= U_p\otimes  e_1\otimes\cdots\otimes e_q
\otimes (dz_1 \wedge \cdots \wedge dz_n)$$
with $U_p$ a holomorphic section of $F$ on $V_p$. 
If in particular $p\in \bigcup\limits_{J\in\mathrm{J}_0}S_J$, we also write 
$$u_{J_p}\big|_{S_{J_p}\cap V_p} = u_p\otimes e^{J_p}\big|_{S_{J_p}\cap V_p}
\otimes (dz_{k+1} \wedge \cdots \wedge dz_{n}),$$
with $u_p$ a holomorphic section of $F|_{S_{J_p}}$ on $S_{J_p}\cap V_p$; if this is the case, since $(U_0)_J=u_J$ for every $J\in \mathrm{J}_0$, we have on $S_{J_p}\cap V_p$ 
\begin{align}\label{res explicit}
U_p(0,\dots,0,z_{k+1}, \cdots, z_n) = u_p(z_{k+1},
\cdots, z_{n}).
\end{align}
We let
\begin{align}\label{Xi}
\Xi^{(p)}(\xi,\eta,x',y'):=\varrho_{l(p)}(\xi,\eta,x',y')
\cdot
\big|U_p(z)\big|_{h_F}^2\,
e^{-\varphi^{J_p}}\, A_p^{-1}.
\end{align}
It is direct to see that there exist bounded functions $\tau_i$, $\upsilon_i$, $\widetilde\tau_i$, and $\widetilde\upsilon_i$ ($i=1,\dots, k$) on $\overline V_p$ such that
\begin{align}\label{jacobi}
e^{-\sum_{j\in J_p}\varphi_j}\, dx\, dy =
\big(\widetilde A_p\big)^{-k}\,\,\Big(1 + \sum_{i=1}^k(\xi_i\tau_i+\eta_i\upsilon_i)\Big) d\xi\, d\eta\, dx'\, dy'
\end{align}
and
\begin{align}\label{Xi-Xi}
\Xi^{(p)}(\xi,\eta,x',y')-\Xi^{(p)}(0,0,x',y') = \sum_{i=1}^k(\xi_i\widetilde\tau_i+\eta_i\widetilde\upsilon_i).
\end{align}

Now we are ready to analyse $I_p(\varepsilon)$.\\
\\
{\bf Claim 1.} $I_p(\varepsilon)\rightarrow 0$ as $\varepsilon\rightarrow 0$ if $p\in P\setminus\bigcup\nolimits_{J\in\mathrm{J}_0}S_J$.
 
Consider the polar coordinate systems in the $\xi_i\eta_i$-directions: 
$$
\xi_i+\sqrt{-1}\,\eta_i=\zeta_i=r_i\, e^{\sqrt{-1}\,\theta_i}\quad (i=1,\dots, k).
$$
By (\ref{e psi}), (\ref{e phi}), (\ref{jacobi}), (\ref{Xi-Xi}), and that both $A_p$ and $\widetilde A_p$ have strictly positive lower bounds on $\overline V_p$, there exists a constant $C_p>0$ independent of $\varepsilon$ such that 
\begin{align*}
\begin{split}
I_p(\varepsilon)&
=
\int_{V_p} \frac{\widetilde\chi(e^{\phi}/\varepsilon)\, \Xi^{(p)}(\xi,\eta,x',y') \, \Big(1 + \sum_{i=1}^k(\xi_i\tau_i+\eta_i\upsilon_i)\Big) d\xi\, d\eta\, dx'\, dy'}{|\zeta_1^2\cdots \zeta_k^2|^{1-\delta}
\widetilde A_p^{\delta k}}\\
\\
&\leqslant
C_p \int_{\prod_i[0,\,R^{1/\alpha_i})} \chi\left(\log\frac{r_1^{2\alpha_1}\cdots r_k^{2\alpha_k}}{\varepsilon}\right)(r_1\cdots r_k)^{2\delta}\, \frac{dr_1}{r_1}\cdots \frac{dr_k}{r_k},
\end{split}
\end{align*}
which converges to $0$ as $\varepsilon\rightarrow 0$ by the dominated convergence theorem, since
$
\chi
$
is compactly supported.
\\
\\
{\bf Claim 2.} If $p\in P\cap \bigcup\limits_{J\in\mathrm{J}_0}S_J$, then
$$
\limsup\limits_{\varepsilon\rightarrow 0} I_p(\varepsilon)\leqslant \frac{B_{| J_p|-1}(\delta)}{\sum\nolimits_{j\in  J_p}\gamma_j}\int_{V_p'}\Xi^{(p)}(0,0,x',y')\,dx'\,dy'.
$$

To see this, note that $\chi$, $\tau_j$, $\upsilon_j$, $\widetilde\tau_j$, $\widetilde\upsilon_j$, and $\Xi^{(p)}$ are all bounded on $\overline V_p$, and hence there exists a constant $C_p>0$ independent of $\varepsilon$ such that
\begin{align*}
\begin{split}
I_p(\varepsilon)
&=
\int_{V_p} \frac{\widetilde\chi(e^{\phi}/\varepsilon)\, \Xi^{(p)}(\xi,\eta,x',y') \, \Big(1 + \sum_{i=1}^k(\xi_i\tau_i+\eta_i\upsilon_i)\Big) d\xi\, d\eta\, dx'\, dy'}{\big(|\zeta_1^2\cdots \zeta_k^2|\big)^{1-\delta}
\big(|\zeta_1|^2+\cdots+|\zeta_k|^2\big)^{\delta k}}\\
&
\leqslant I_p^{(0)}(\varepsilon)\, \int_{V_p'}\Xi^{(p)}(0,0,x',y')\,dx'\,dy'+2C_p\, I_p^{(1)}(\varepsilon)\, \int_{V_p'}\,dx'\,dy'
\end{split}
\end{align*}
where
$$
I_p^{(m)}(\varepsilon):=\int_{B_R(0)} 
 \frac{\widetilde\chi(e^{\phi}/\varepsilon) \big(|\zeta_1|^2+\cdots+|\zeta_k|^2\big)^{\frac{m}{2}}\, d\xi\, d\eta}{\big(|\zeta_1^2\cdots \zeta_k^2|\big)^{1-\delta}
\big(|\zeta_1|^2+\cdots+|\zeta_k|^2\big)^{\delta k}}\quad (m=0,1).
$$
We consider the following ``spherical coordinate system'' in the total $\xi\eta$-direction: 
$$
\xi_i+\sqrt{-1}\,\eta_i=\zeta_i=\sigma_{i}^{\frac{1}{2}}\, r\, e^{\sqrt{-1}\,\theta_i}\quad (i=1,\dots,k;\ \sigma_k :=1-\sigma_1-\cdots -\sigma_{k-1}),
$$
with  
$$
r\in [0,1),\,(\sigma_1,\dots,\sigma_{k-1})\in\Delta_{k-1},\text{ and } (\theta_1,\dots,\theta_k)\in [0,2\pi]^k
$$ 
where
$$
\Delta_{k-1}=\big\{(\sigma_1,\dots,\sigma_{k-1})\in [0,\infty)^k\,\big|\, \sigma_1+\cdots+\sigma_{k-1}\leqslant 1\big\}.
$$
A direct commutation shows that
$$
|\zeta_1^2\cdots\zeta_k^2|^{1-\delta}\, \big(|\zeta_1|^2+\cdots+|\zeta_k|^2\big)^{\delta k} 
= 
(\sigma_1\cdots\sigma_{k-1}\,\sigma_k)^{1-\delta}\,r^{2k},
$$
$$
e^{\phi}= 
\sigma_1^{\alpha_1}\cdots\sigma_{k-1}^{\alpha_{k-1}}\,\sigma_k^{\alpha_k}\,r^{2(\alpha_1+\cdots+\alpha_k)},
$$
and
$$
d\xi\, d\eta=\frac{1}{2^{k-1}}r^{2k}\, \frac{dr}{r}\,d\sigma_1\cdots d\sigma_{k-1}\,
d\theta_1\cdots d\theta_k.
$$
Thus
\begin{align*}
I_p^{(m)}(\varepsilon)
&=
2\pi^k\int_{\Delta_{k-1}}\left( \int_0^R
\chi\left(\log\frac{(\prod_i\sigma_i^{\alpha_i})r^{2\sum_i\alpha_i}}{\varepsilon}\right)\,
r^m\,\frac{dr}{r}\right)
\frac{d\sigma_1\cdots d\sigma_{k-1}}{(\sigma_1\cdots\sigma_{k-1}\,\sigma_k)^{1-\delta}}.
\end{align*}
Note that
$$
\int_{\Delta_{k-1}}\frac{d\sigma_1\cdots d\sigma_{k-1}}{(\sigma_1\cdots\sigma_{k-1}\,\sigma_k)^{1-\delta}}<\infty\ \text{ for every }\ \delta\in(0,1].
$$
Since $\chi$ is compactly supported,
$$
\chi\left(\log\frac{(\prod_i\sigma_i^{\alpha_i})r^{2\sum_i\alpha_i}}{\varepsilon}\right)\rightarrow 0\ \text{ as }\ \varepsilon\rightarrow 0
$$
for almost every $(\sigma_1,\dots,\sigma_{k-1},r)\in\Delta_{k-1}\times [0,R)$, and hence
$I_p^{(1)}(\varepsilon)\rightarrow 0$ as $\varepsilon\rightarrow 0$ by the dominated convergence theorem. Finally,
\begin{align*}
\begin{split}
I_p^{(0)}(\varepsilon)
&=
2\pi^k\int_{\Delta_{k-1}}\left( \int_0^R
\chi\left(\log\frac{(\prod_i\sigma_i^{\alpha_i})r^{2\sum_i\alpha_i}}{\varepsilon}\right)\frac{dr}{r}\right)
\frac{d\sigma_1\cdots d\sigma_{k-1}}{(\sigma_1\cdots\sigma_{k-1}\,\sigma_k)^{1-\delta}}\\
&
=
\frac{\pi^k}{\sum_i\alpha_i}\int_{\Delta_{k-1}}\left( \int_0^{\frac{(\prod_i\sigma_i^{\alpha_i})R^{2\sum_i\alpha_i}}{\varepsilon}}
\chi(\log v)\frac{dv}{v}\right)
\frac{d\sigma_1\cdots d\sigma_{k-1}}{(\sigma_1\cdots\sigma_{k-1}\,\sigma_k)^{1-\delta}}\\
&\quad\longrightarrow \frac{\pi^k}{\sum_i\alpha_i}\int_{\Delta_{k-1}}
\frac{d\sigma_1\cdots d\sigma_{k-1}}{(\sigma_1\cdots\sigma_{k-1}\,\sigma_k)^{1-\delta}}\ \text{ as }\ \varepsilon\rightarrow 0
\end{split}
\end{align*}
by the monotone convergence theorem (
$
\int_0^\infty
\chi(\log v)\,\frac{dv}{v}=\int_\mathbf R\chi=1
$).
This completes the proof of Claim 2.

In summary, 
\begin{align*}
\begin{split}
&\limsup_{\varepsilon\rightarrow 0}I(\varepsilon)\leqslant 
\sum\limits_{p\in P\cap\bigcup\nolimits_{J\in\mathrm{J}_0}S_J}
\frac{B_{| J_p|-1}(\delta)}{\sum\nolimits_{j\in J_p}\gamma_j}\int_{V_p'}\Xi^{(p)}(0,0,x',y')\,dx'\,dy'.
\end{split}
\end{align*}
By (\ref{Ap}), (\ref{res explicit}), and (\ref{Xi}),
\begin{align*}
&
\int_{V_p'}\Xi^{(p)}(0,0,x',y')\,dx'\,dy'\\
&\quad=
\int_{V_p'} \varrho_{l(p)}(0,0,x',y')\frac{ 
\big|u_p(z_{k(p)+1},\dots,z_n)\big|_{h_F}^2\,
e^{-\varphi^{J_p}(0,0,x',y')}
}{
 A_p(0,0,x',y')
}
dx'\, dy'\\
&\quad=
\int_{S_J\cap V_p}
\varrho_{l(p)}
\frac{\langle u_{J_p}\ra^2_{h^{J_p}\otimes h_F}}{
|\sigma^{J_p}|_{h^{J_p}}^{2(1-\delta)}
\prod\limits_{J\in\mathrm{J}_0\setminus\{J_p\}}\,\Big(\sum \limits_{j\in J}|\sigma_j|^2_{h_j}\Big)^{\delta | J|}}.
\end{align*}
Therefore,
\begin{align*}
\begin{split}
&\limsup\limits_{\varepsilon\rightarrow 0}I(\varepsilon)\leqslant 
\\
&\quad\quad 
\sum\limits_{J\in\mathrm{J}_0}\frac{B_{| J|-1}(\delta)}{\sum\limits_{j\in J}\gamma_j}
\int_{S_J} 
\frac{|\sigma^J|_{h^J}^{2\delta_\bullet}}{
\prod\limits_{J'\in\mathrm{J}_0\setminus\{J\}}\,\Big(\sum \limits_{j\in J'}|\sigma_j|^2_{h_j}\Big)^{\delta |J'|}}
\left\langle \frac{u_J}{\sigma^J}\right\ra^2_{h_F}.
\end{split}
\end{align*}
\end{proof}

\section{\bf Proof of Theorem \ref{from snc}} \label{app}
In the following we assume that $\mathcal W$ is an $\mathcal S$-snc family. We first prove a lemma which enables extension of sections satisfying certain vanishing conditions.

Consider the restriction maps:
\begin{align*}
\xyz{
\Gamma(X,K_X\otimes S\otimes F)\ar[r]^-\alpha
&
\Gamma\big(\underline{\mathcal W},(K_X\otimes S\otimes F)|_{\underline{\mathcal W}}\big)\ar[d]^-\beta\\
&
\Gamma\big(\underline{\mathcal W}',(K_X\otimes S\otimes F)|_{\underline{\mathcal W}'}\big).
}
\end{align*}
\begin{lem}\label{van ext}
(1) $\ker\,\beta\subseteq \mathrm{im}\,\alpha$ if $(F,h_F)$ is a holomorphic vector bundle with a smooth hermitian metric such that $\sqrt{-1}\, \Theta_{h_F}\succcurlyeq\pm\sqrt{-1}\, \Theta_{h_j}\otimes\mathrm{id}_F$ for $j=1,\dots,q$.\\
(2) Suppose that $(F,h_F)$ is a holomorphic line bundle with a singular hermitian metric such that $\sqrt{-1}\, \Theta_{h_F}\succcurlyeq\pm\sqrt{-1}\, \Theta_{h_j}$ for $j=1,\dots,q$. Given $u\in \ker\,\beta$, if $\mathcal I(h_F|_W)=\mathcal O_W$ for every $W\in\mathcal W$ along which $u$ is not identically $0$,
then there exists $U\in \Gamma\big(X,\mathcal I(h_F)(K_X\otimes S\otimes F)\big)$ such that $U|_{\underline{\mathcal W}}=u$.
\end{lem}

\begin{rmk}\label{shrink stein} Let $\xyz{X\ar[r]^-\pi&T}$ be a projective morphism to a Stein space given by assumption. Consider the statement of Lemma \ref{van ext}. The canonical maps $\alpha$ and $\beta$ are obtained by taking global sections of the following canonical morphisms of $\mathcal O_T$-modules:
$$
\xyz{\pi_*\mathcal O_X(K_X\otimes S\otimes F)\ar[r]^-{\rho}&
\pi_*\mathcal O_X(K_X\otimes S\otimes F)\otimes_{\mathcal O_X}\mathcal O_X/\mathcal I_{\underline{\mathcal W}}\ar[d]^-{\rho^{(1)}}\\
&\pi_*\mathcal O_X(K_X\otimes S\otimes F)\otimes_{\mathcal O_X}\mathcal O_X/\mathcal I_{\underline{\mathcal W'}}.}
$$
Proving Lemma \ref{van ext} (1), for example, amounts to showing that $\mathrm{ker}\, \rho^{(1)}(T)\subseteq \mathrm{im}\, \rho(T)$; since $T$ is Stein, this is further equivalent to show that $\mathrm{ker}\, \rho^{(1)}\subseteq \mathrm{im}\, \rho$ or, equivalently, that $\mathrm{ker}\, \rho^{(1)}(T')\subseteq \mathrm{im}\, \rho(T')$ for every relatively compact open subset $T'$ of $T$.
\end{rmk}

\begin{proof}[Proof of Lemma \ref{van ext}]
Suppose we have\footnote{Since $\underline{\mathcal W}$ is reduced and $\underline{\mathcal W}$ is an snc family, a section 
$
u\in \Gamma\big(\underline{\mathcal W},(K_X\otimes S\otimes F)|_{\underline{\mathcal W}}\big)
$
can always be represented uniquely by its ``restrictions'' (via adjunction)
$
u_W\in \Gamma(W,K_W\otimes S^W\otimes F|_W)\, (W\in\mathcal W)
$.} any $u=(u_W)_{W\in\mathcal W}\in \ker\,\beta$. 

We will apply Theorem \ref{ext snc} with a particular type of function $G$. As pointed out in \cite{gz} (in a slightly different manner), a function $G\in C^{\infty}(-A,\infty)$ lies in $\mathcal G_A$ if
\begin{itemize}
\item[-] $G^{(j)}(-A^+)$ ($j=0,1,2$) all exist and are strictly positive,
\item[-] $G'(-A^+)^2-G(-A^+)G''(-A^+)\geqslant 0$,
\item[-] $G''$ is strictly positive and nonincreasing, and
\item[-] $G'(\infty):=\lim\limits_{y\rightarrow \infty}G'(y)$ exists and is strictly positive.
\end{itemize}
For example, given a number $\varepsilon>0$ and a smooth function $H$ on $(-A,\infty)$ which is strictly positive, nonincreasing, and integrable with $H(-A^+)\in(0,\infty)$, one can construct a function $G_{\varepsilon, H}\in\mathcal G_A$ as follows. Let 
$$
G_1(y) := \frac{1}{\varepsilon} + \int_{-A^+}^y H(t) dt\ \text{ and }\ G_{\varepsilon,H}(y):=\frac{1}{\varepsilon^2H(-A^+)} + \int_{-A^+}^y G_1(t)dt.
$$
It is direct to see the four conditions above hold for $G_{\varepsilon,H}$, and hence $G_{\varepsilon,H}\in\mathcal G_A$. We have
$$
G_{\varepsilon,H}'(\infty)=\frac{1}{\varepsilon} + \int_{-A^+}^\infty H(t) dt
$$
and 
$$
\frac{G_{\varepsilon,H}'\big(-\log|\sigma_S|_{h_S}^{2\gamma_\bullet}\big)}{G_{\varepsilon,H}\big(-\log|\sigma_S|_{h_S}^{2\gamma_\bullet}\big)}
\leqslant 
\frac{G_{\varepsilon,H}'(A^+)}{G_{\varepsilon,H}(-A^+)}=\varepsilon H(-A^+).
$$
Now we examine curvature conditions (\ref{cur snc 1}), (\ref{cur snc 2}), and (\ref{cur snc'}) with $G=G_{\varepsilon,H}$. Consider the following family of conditions:
\begin{align}\label{cur snc}
\begin{split}
&\Theta_{h_F}
-
\sum_{j\in J(\mathcal W)}
\big(
\delta(t_j-1)
+\varepsilon H(-A^+)s\gamma_j
\big)
\,\Theta_{h_j}\otimes\mathrm{id}_F
\\ 
&\quad\quad\quad\quad\quad\quad\quad -
\sum_{j\notin J(\mathcal W)}
\big(
\varepsilon H(-A^+)s\gamma_j
-\delta_j
\big)
\,\Theta_{h_j}\otimes\mathrm{id}_F
\geqslant_{\mathrm{Nak}} 0
\\
&\text{ for }\ s\in [0,1]\ \text{ and }\ (t_j)\in \prod\limits_{j\in J(\mathcal W)} [0, T_j]\quad \bigg(T_j:=\sum\limits_{W\,:\,j\in J_W}|J_W|\bigg).
\end{split}
\end{align}
When $h_F$ is smooth, we see that (\ref{cur snc 1}) and (\ref{cur snc 2}) hold by taking $s=0$ and $s=1$ in (\ref{cur snc}); when $F$ is a line bundle and $h_F$ is a singular metric, (\ref{cur snc'}) is equivalent to (\ref{cur snc}). (\ref{bdd below}) holds if we further take 
$$
H(y)=y^{-2},\quad 
\varepsilon =1,\quad
\gamma_j=\left\{
\begin{array}{cc}
1,& j\in J(\mathcal W);\\
0, &j\notin J(\mathcal W),
\end{array}
\right.
\quad
\delta_j=\left\{
\begin{array}{cc}
\delta,& j\in J(\mathcal W);\\
0, & j\notin J(\mathcal W).
\end{array}
\right.
$$
It remains to show that the integrability condition (\ref{snc int finite})
$$
\int_{W} 
\frac{|\sigma^{J_W}|_{h^{J_W}}^{2\delta}}{
\prod\limits_{J'\in\mathrm{J}_0(\mathcal W)\setminus\{J_W\}}\,\Big(\sum \limits_{j\in J'}|\sigma_j|^2_{h_j}\Big)^{\delta
\,|J'|}}\left\langle \frac{u_W}{\sigma^{J_W}}\right\ra^2_{h_F}
<\infty\quad (W\in\mathcal W)
$$
holds for some $0<\delta<1$. We are allowed to replace $X$ with a suitable relatively compact open subset, as indicated in Remark \ref{shrink stein}, and hence the verification of (\ref{snc int finite}) is purely local. For a point $p\in W$ suppose (possibly after reordering the $S_j$) that 
$$
\{j\in J(\mathcal W)\setminus J_W\,|\, p\in S_j\}=\{1,\dots,k\}
$$ 
and 
$$
\{j\in \{1,\dots,n=\dim X\}\setminus J_W\, |\,p\in S_j\} = \{1,\dots,k,k+1,\dots,k'\},
$$
and that $(V_p,\{z_j\})$ is a coordinate patch with center $p$ on which $F$ and all $S_j$ are trivial and $\sigma_j$ is represented by $z_j$ for $j=1,\dots,k'$. Express $u_W$ as a holomorphic vector-valued function $(u_1,\dots,u_r)$ and $h_F$ is a measurable hermitian matrix $h_{a\overline b}$ of order $r$. It suffices to show that, for a fixed set of positive integers 
$$
m_{j_1\dots j_m}\quad (1\leqslant j_1\leqslant\cdots\leqslant j_l\leqslant k\ \text{ and }\ 1\leqslant l\leqslant k)
$$
we have
$$
\int_{V_p\cap W} 
\frac{\sum\limits_{a,b}h_{a\overline b}u_a \overline u_b}{
\prod\limits_{l=1}^k \prod\limits_{1\leqslant j_1\leqslant\cdots\leqslant j_l\leqslant k}\Big(\sum\limits_{m=1}^l |z_{j_m}|^2\Big)^{\delta
\,l\, m_{j_1\cdots j_l}}
|z_1\cdots z_k z_{k+1}\cdots z_{k'}|^2} <\infty
$$
for sufficiently small $\delta>0$. Since $u\in\ker\, \beta$, $u_1,\dots, u_r$ all vanish along the zero set of the function $z_1\cdots z_{k'}$, i.\ e., there exist holomorphic functions $v_1,\dots, v_r$ on $V$ such that 
$
u_a = z_1\cdots z_{k'}\,v_a\, (a=1,\dots,r)
$, and hence the integral becomes
$$
\int_{V_p\cap W} 
\frac{\sum\limits_{a,b}h_{a\overline b}v_a \overline v_b}{
\prod\limits_{l=1}^k \prod\limits_{1\leqslant j_1\leqslant\cdots\leqslant j_l\leqslant k}\Big(\sum\limits_{m=1}^l |z_{j_m}|^2\Big)^{\delta
\,l\, m_{j_1\cdots j_l}}}.
$$
If $h_F$ is smooth, the integral is clearly finite for sufficiently small $\delta>0$. In the case of a line bundle with a singular metric, the integral is finite for sufficiently small $\delta>0$ by the validity of the strong openness conjecture.
\end{proof}

Now we can proceed the proof of Theorem \ref{from snc}, which is a slight modification from a strategy suggested by Demailly. Suppose that $\mathcal W^{(m)}\supsetneq( \mathcal W^{(m+1)})=\emptyset$. Consider a section $u\in \Gamma\big(\underline{\mathcal W},(K_X\otimes S\otimes F)|_{\underline{\mathcal W}}\big)$
and its restrictions:
$$
\xyz{
\Gamma(X,K_X\otimes S\otimes F)\ar[r]
\ar[rd]\ar[rddd]
&
\Gamma\big(\underline{\mathcal W},(K_X\otimes S\otimes F)|_{\underline{\mathcal W}}\big)\ar[d]
&
u\ar@{|->}[d]
\\
&
\Gamma\big(\underline{\mathcal W}^{(1)},(K_X\otimes S\otimes F)|_{\underline{\mathcal W}^{(1)}}\big)\ar[d]
&
u_1\ar@{|->}[d]
\\
&\vdots \ar[d]&\vdots \ar@{|->}[d]\\
&
\Gamma\big(\underline{\mathcal W}^{(m)},(K_X\otimes S\otimes F)|_{\underline{\mathcal W}^{(m)}}\big) 
&
u_m.
}
$$
By Lemma \ref{van ext}, $u_m=U_m|_{\underline{\mathcal W}^{(m)}}$ for some $U_m\in \Gamma(X,K_X\otimes S\otimes F)$. In particular, $u_{m-1}-U_m|_{\mathcal S_{\mathcal W^{(m)}}}=0$. Suppose we have created 
$$
U_m,\dots, U_{m-k+1}\in \Gamma(X,K_X\otimes S\otimes F)
$$
such that
$$
u_{m-k}-(U_m+\cdots+U_{m-k+1})|_{\mathcal \underline{\mathcal W}^{(m-k+1)}}=0.
$$
Lemma \ref{van ext} again yields some $U_{m-k}\in \Gamma(X,K_X\otimes S\otimes F)$ such that
$$
u_{m-k}-(U_m+\cdots+U_{m-k+1})|_{\mathcal \underline{\mathcal W}^{(m-k)}}=U_{m-k}|_{\mathcal \underline{\mathcal W}^{(m-k)}}.
$$
Repeating this procedure we will reach the stage $k=m$ and obtain the desired extension. In the case of a line bundle with a singular metric, all the sections $U_j$, and hence $U$, may be chosen in $\Gamma\big(X,\mathcal I(h_F)(K_X\otimes S\otimes F)\big)$.

\appendix

\section{\bf Preliminaries for solving variants of the $\overline\partial$-equation}\label{d-bar prel}  

\begin{thm} [(6.1) Theorem in \cite{agbook}] \label{L2 est tool} Let $(X,g)$ be an $n$-dimensional K\"ahler manifold admitting a complete K\"ahler metric (which is not necessarily $g$), $(E,h)$ a holomorphic vector bundle on $X$ with a smooth hermitian metric, and $\lambda$ and $\mu$ two strictly positive {\it bounded} smooth functions. Let 
$$
R=R_{h,\lambda,\mu}:=\sqrt{-1}\Bigg(\lambda \Theta_{h} - \bigg(\partial \overline{\partial}
\lambda - \frac{1}{\mu}\partial{\lambda}\wedge\overline\partial{\lambda}\bigg)\otimes\mathrm{id}_E\Bigg).
$$
For any $\eta\in L^2_{g,h}(X, \wedge^{n,q}T^*X\otimes E)$ (with $q\geqslant 1$) and any nonnegative constant $C$, if $CR$ is semipositive, $\overline\partial \eta=0$, and
\begin{align}\label{appriori for twisted dbar}
|(\eta, v)_{L^2_{g,h}}|^2
\leqslant 
C
\int_X
R[v,v]_{g,h}\, dV_g
\end{align}
for every $v\in \mathcal D^{n,q}(X,E)$, then there exist a locally Lebesgue integrable $E$-valued $(n,q-1)$-form $\gamma$ and a locally integrable $E$-valued $(n,q)$-form $\beta$ such that 
$$
\overline\partial \gamma = \eta\ \text{ in the sense of current}
$$
and
$$  
\int_X \frac{1}{\lambda+\mu}\,|\gamma|^2_{g,h}\,dV_g\leqslant C.
$$
\end{thm}


We need the following slight modification of Lemma 4.6 of \cite{gz}:  
\begin{lem} []\label{loc L2 conv => loc unif conv} Let $Z$ be an analytic subset of a hermitian manifold $(M,g)$. Suppose that there are 
\begin{itemize}
\item[(1)] open subsets  
$
O_1\subseteq O_2\subseteq\cdots\subseteq  O_m
$
of $M$ such that 
$M\setminus Z=\bigcup\limits_m O_m$,
\item[(2)] a sequence $U_m$ of measurable sections of $K_M\otimes E$, $E$ being a hermitian holomorphic vector bundle equipped with a smooth hermitian metric $h$, for every point $p\in M\setminus (Z)_{\mathrm{sing}}$ there exists an open neighborhood $V_p$ of $p$ and an index $m_p$ such that $U_m|_{V_p}\, (m\geqslant m_p)$ are all holomorphic, and 
\item[(3)] a sequence $w_m$ of positive Lebesgue measurable functions  on $M$ such that for every compact subset $K$ of $M\setminus Z$ there exists an index $m_K$ such that the family of functions $w_m\, (m\geqslant m_K)$ are uniformly bounded away both from $0$ and from $\infty$, and
\end{itemize}
If $w_m$ converges to a function $w$ almost everywhere, and if
$
\liminf\limits_{m\rightarrow\infty} \int_{O_m}
 w_m\langle U_m\ra^2_h
$
exists as a real number, then $U_m$ admits a subsequence which converges uniformly on every compact subset of $M$ to a section $U\in\Gamma\big(M,\mathcal O_M(K_M\otimes E)\big)$  such that
$$
\int_M w\langle U\ra^2_h
\leqslant 
\liminf\limits_{m\rightarrow\infty} \int_{O_m}
w_m\langle U_m\ra^2_h.
$$
\end{lem}
\begin{proof} We may assume that $\int_{O_m}w_m\langle U_m\ra^2_h$ actually converges as $m\rightarrow\infty$ by passing to subsequences, and that $Z$ is a submanifold by considering $(M\setminus Z_{\mathrm{sing}},Z\setminus Z_{\mathrm{sing}})$ instead of $(M,Z)$ and then extending the obtained section on $M\setminus Z_{\mathrm{sing}}$ to $M$ by Riemann's extension theorem. Besides, it suffices to show, as we will do in the next paragraph, that every point $p\in M$ admits an open neighborhood $N_p$ on which $U_m$ with $m$ sufficiently large form a normal family. 
More precisely, were this done, then $M$ is covered by countably many such open sets $N_{p_1}, N_{p_2},\dots$, and an application of the diagonal method yields a subsequence $U_{m_k}$ which converges uniformly on every compact subset of $M$ to a limit $U$, which is clearly a holomorphic section of $K_M\otimes E$ on $M$ by the condition (2). Then for every $j$ we have
$$
\int_{O_j} w\langle U\ra^2_h\leqslant \lim\limits_{k\rightarrow\infty}\int_{O_{m_k}} w_{m_k}\langle U_{m_k}\ra^2_h
$$
by Fatou's lemma, and hence 
$$
\int_M w\langle U\ra^2_h=\int_{M\setminus Z} w\langle U\ra^2_h=\lim\limits_{j\rightarrow\infty}\int_{O_j} w\langle U\ra^2_h
\leqslant \lim\limits_{m\rightarrow\infty}\int_{O_m} w_m\langle U_m\ra^2_h.
$$
Therefore it remains to find the desired neighborhood $N_p$ for every $p\in M$. 

We may assume that $M$ is an open subset of $\mathbf C^n$. It suffices to find a neighborhood $N_p$ for every $p\in M$ such that $U_m$ for $m$ sufficiently are uniformly bounded with respect to the euclidean $L^2$ norms on $N_p$. For $p\in M\setminus Z$ such $N_p$ exists obviously since both $w_m$ and the metric $h$ have strictly positive uniform lower and upper bounds. For $p\in Z$, we may further assume that 
$$
(M, Z,p)=\big(D_2(0)^n, \{0\}\times D_2(0)^{n-1}, (0,\dots,0)\big)
$$
where $D_R(0):=\{z\in\mathbf C\,|\, |z|<R\}\big)$, and that $E$ is the trivial bundle. Having no control on $w_m$, we need to apply the following elementary fact (cf.\ Lemma 4.4 of \cite{gz}): for fixed $0<r<1$ and for every holomorphic function $F$ on $D_2(0)^n$, we have
$$
\int_{D_1(0)^n} |U|^2\, d\lambda
\leqslant 
\frac{1}{1-r^2} \int_{\big(D_1(0)\setminus D_r(0)\big)\times D_1(0)^{n-1}} |U|^2\, d\lambda.
$$ 
For $m$ sufficiently large $\big(D_1(0)\setminus D_r(0)\big)\times D_1(0)^{n-1}$ is covered by $O_m$, and hence 
$$
\int_{\big(D_1(0)\setminus D_r(0)\big)\times D_1(0)^{n-1}} |U|^2\, d\lambda
$$ 
is bounded by a fixed multiple of $\int_{O_m}w_m\langle U_m\ra_h$. This completes the proof.
\end{proof}

\end{document}